\documentclass[letterpaper, 10 pt, conference]{ieeeconf}  % Comment this line out if you need a4paper

\IEEEoverridecommandlockouts                              % This command is only needed if 
\usepackage{graphics} % for pdf, bitmapped graphics files
\usepackage{epsfig} % for postscript graphics files
\usepackage{mathptmx} % assumes new font selection scheme installed
\DeclareMathAlphabet{\mathcal}{OMS}{cmsy}{m}{n}
\usepackage{amsmath} % assumes amsmath package installed
\usepackage{amssymb}  % assumes amsmath package installed

\usepackage[ruled,vlined,linesnumbered]{algorithm2e}
\SetAlgoLined

\usepackage{multirow}
\usepackage{subcaption}


\usepackage{amsthm}

\theoremstyle{plain}
\newtheorem{theorem}{Theorem}

\newtheorem{lemma}{Lemma}

\newtheorem{assumption}{Assumption}

\DeclareMathOperator{\Tr}{Tr}

\title{\LARGE \bf
Policy Iteration for Domain Randomized Linear Quadratic Systems
}

\author{Abbas Pasdar$^{1}$ and Farnaz Adib Yaghmaie$^{2}$% <-this % stops a space
\thanks{Abbas Pasdar is supported by Sensor Informatics and Decision-making for the Digital Transformation (SEDDIT). Farnaz Adib Yaghmaie is supported by the Excellence Center at Linköping–Lund in Information Technology (ELLIIT), ZENITH, and partially by Sensor Informatics and Decision-making for the Digital Transformation (SEDDIT).  This work was partly performed within the Competence Center SEDDIT, supported by Sweden’s Innovation Agency within the research and innovation program Advanced Digitalization.}% <-this % stops a space
\thanks{$^{1}$Department of Electrical Engineering, Link\"{o}ping University, Link\"{o}ping, Sweden
        {\tt\small abbas.pasdar@liu.se}}%
\thanks{$^{2}$Department of Electrical Engineering, Link\"{o}ping University, Link\"{o}ping, Sweden {\tt\small farnaz.adib.yaghmaie@liu.se}}%
}

\begin{document}

\maketitle
\thispagestyle{empty}
\pagestyle{empty}

%%%%%%%%%%%%%%%%%%%%%%%%%%%%%%%%%%%%%%%%%%%%%%%%%%%%%%%%%%%%%%%%%%%%%%%%%%%%%%%%

\begin{abstract}
In this work, we study policy optimization under domain randomization for linear quadratic control, focusing on learning a single state-feedback controller that minimizes the average cost across systems with uncertain dynamics.  We propose a policy iteration algorithm with a step-size rule that preserves stability across all sampled systems at each iteration.  We show that the method yields monotonic improvement of the sample-average objective and that a stabilizing step size always exists. Under standard smoothness assumptions, the iterates converge subsequentially to stationary points, and under a gradient-dominance condition, we obtain a global linear convergence rate.
\end{abstract}

% \begin{keywords}%
%   Semidefinite programming, linear quadratic regulation, domain randomization
% \end{keywords}
%%%%%%%%%%%%%%%%%%%%%%%%%%%%%%%%%%%%%%%%%%%%%%%%%%%%%%%%%%%%%%%%%%%%%%%%%%%%%%%%

\section{Introduction}\label{sec_intro}
One of the main challenges in Reinforcement Learning (RL) and control is to learn policies that generalize across environments with uncertain dynamics \cite{dulac2019challenges,pinto2017robust}. This issue is particularly prominent in robotics and sim-to-real transfer, where discrepancies between simulated and real-world dynamics can lead to significant performance degradation \cite{tobin2017domain,koos2013transferability}. 
A widely used approach to mitigate this problem is \emph{Domain Randomization} (DR), in which policies are trained across a distribution of environments obtained by varying system parameters \cite{openai2019rubiks}. This is typically done in simulation, where it is possible to efficiently sample from a wide range of dynamics and generate large amounts of data, but this approach generally lacks theoretical guarantees.

On the other hand, in control theory, robust control methods have been developed to design controllers that maintain stability and performance under model uncertainty \cite{zhou1998essentials}. Robust control typically focuses on worst-case performance guarantees, which can be conservative in practice. In contrast, domain randomization aims to optimize average performance across a distribution of systems; however, it lacks the rigorous stability and convergence guarantees that are central to control.

Linear quadratic regulation (LQR) offers a natural and analytically tractable setting for investigating generalization and robustness in control. Beyond its classical role in optimal control, LQR has also emerged as a canonical benchmark for studying generalization and robustness in control, owing to its rich structure and amenability to theoretical analysis \cite{tu2019gap,Adib2021Linear,Adib2018Output}. Recently, DR problems have been studied in the LQR setting, where the aim is to learn a single linear state-feedback controller that minimizes the average quadratic cost across a distribution of linear systems with uncertain dynamics. In \cite{fujinami2025pg}, Policy Gradient (PG) methods were used in this setting, providing initial insights into robustness and generalization across system variations. In \cite{AbbasECC2026}, semidefinite programming (SDP) approaches have been proposed to solve the DR-LQR problem with the main benefit of handling additional constraints, such as stability-guarantee constraints, at each iteration. However, the computational cost of solving SDP problems grows substantially with the system dimension, which can make SDP-based approaches computationally prohibitive for high-dimensional systems.

Policy gradient approaches in the LQR setting require either the model of dynamics or long trajectories of system rollouts to estimate the policy gradient \cite{fazel2018global}. In contrast, Policy Iteration (PI) methods exploit the Bellman operator to learn the optimal policy without relying on long rollouts and naturally extend to data-driven setups \cite{bertsekas2012dynamic,lewis2012optimal}. In practice, PI methods are known to converge faster than PG methods, and they can be interpreted as a second-order optimization method \cite{fazel2018global}. In addition, PI methods build the backbone of more modern actor-critic algorithms, which have been successfully applied to a wide range of control problems \cite{lillicrap2015continuous, schulman2015trust, mnih2016asynchronous}. 

Policy iteration in the LQR setting has been widely studied in the literature, including robustness analysis \cite{pang2021robust}, mean-field control problems \cite{li2024policy}, stochastic optimal control \cite{li2022stochastic}, model-free and data-driven implementations, and indirect/direct learning architectures \cite{song2024role,forni2023onpolicy,yang2023lambda}. These developments, however, largely concern the \emph{single-system} setting and do not address the domain-randomized setting considered here.

In this paper, we study the domain-randomized LQR problem through the lens of policy iteration. We derive policy evaluation and policy improvement steps for the sample-average objective and extend the classical policy iteration framework to the multi-system setting.
In particular, we use the Lagrange multiplier method to formulate the constrained optimization problem arising from the DR‑LQR formulation, and we show that this leads to a PI representation of the DR‑LQR problem.
Under moderate heterogeneity, we show that a stabilizing step size always exists and that the resulting updates yield a monotonic decrease of the sample-average objective. Furthermore, under standard smoothness assumptions, we establish subsequential convergence to stationary points, and under a gradient-dominance condition, we obtain a linear convergence rate. Overall, the main contribution of this paper is the development of a policy iteration algorithm for domain-randomized LQR together with theoretical analysis on stability preservation, descent, and convergence.

The remainder of the paper is organized as follows. 
In Section~\ref{sec_background}, we present the necessary preliminaries. 
Section~\ref{Sec_OPPI} derives the policy iteration scheme from an optimization perspective. 
In Section~\ref{sec_DRPI}, we present the proposed policy iteration algorithm for domain-randomized LQR and discuss its implementation. 
Theoretical analysis, including stability preservation, descent, and convergence results, are established in Section~\ref{sec_Theory}. 
Numerical experiments illustrating the performance of the proposed method are provided in Section~\ref{sec_experiment}. 
Finally, Section~\ref{sec_conclusion} concludes the paper.

%================================================================
%================================================================
%================================================================

\section{Preliminaries}
\label{sec_background}
\textbf{Notations:}
Let $\mathbb{R}^{m \times n}$ denote the set of real matrices of size $m \times n$. 
For a matrix $A$, $\|A\|_F$ and $\|A\|$ denote the Frobenius and spectral norms, respectively, and $A^\top$ denotes transpose. 
The Frobenius inner product is denoted by $\langle \cdot, \cdot \rangle$. 
The spectral radius of $A$ is defined as $\rho(A) := \max \{ |\lambda| : \lambda \in \mathrm{spec}(A) \}$. 
The vectorization $\mathrm{vec}(A)$ stacks the columns of $A \in \mathbb{R}^{m \times n}$ into a vector in $\mathbb{R}^{mn}$. 
For a symmetric matrix $P \in \mathbb{R}^{n \times n}$, $P \succ 0$ ($P \succeq 0$) denotes that $P$ is positive definite (semidefinite). 
The Kronecker product is denoted by $\otimes$. 
For a positive integer $M$, we use $[M] := \{1, \dots, M\}$.
The support of a probability distribution $\mathcal D$ is denoted as $\text{supp}(\mathcal D)$.

\subsection{Linear Quadratic Regulator (LQR)}
We consider the discrete-time linear system
\[x_{t+1} = A x_t + B u_t + w_t,\]
where $x_t \in \mathbb{R}^n$ is the state, $u_t \in \mathbb{R}^m$ is the control input (action), and $w_t \in \mathbb{R}^n$ is the process noise. 
The stage cost is quadratic
\[c(x_t,u_t) = x_t^\top Q x_t + u_t^\top R u_t,\]
where $Q \succeq 0$ and $R \succ 0$.
Under standard stabilizability and detectability assumptions, the optimal policy is linear $u_t = K_\star x_t$ and can be obtained from the discrete-time algebraic Riccati equation \cite{anderson2007optimal}.
\subsection{Dynamic Programming for LQR}

A fundamental dynamic programming method for solving the infinite-horizon
discrete-time LQR problem is policy iteration, which alternates between
policy evaluation and policy improvement \cite{anderson2007optimal}.

Given a stabilizing controller $K_i$ at iteration $i$, the value matrix $P_i$ is obtained as the unique positive semidefinite solution of the Lyapunov equation
\begin{equation}
P_i = Q + K_i^\top R K_i + (A + B K_i)^\top P_i (A + B K_i),
\end{equation}
which is called the \emph{policy evaluation} step. The \emph{policy improvement} step updates the controller according to
\begin{equation} \label{eq_PolicyImprove_single}
K_{i+1} = -(R + B^\top P_i B)^{-1} B^\top P_i A.
\end{equation}
Under standard stabilizability and detectability assumptions, this iteration
is well-defined and converges to the optimal controller $K_\star$
\cite{bertsekas2012dynamic,hewer1971iterative,gravell2021midpoint}.

\subsection{Optimization-Based LQR}

Beyond dynamic programming, the LQR problem can be viewed as a nonconvex optimization problem over the linear feedback gain $K$.  
For a stabilizing controller $K$, the infinite-horizon cost can be written as \cite{AbbasECC2026}
\begin{equation} \label{eq_LQRPO}
    \begin{aligned}
        \min_{K\in\mathcal{K},\, \Sigma}\quad &J(K) = \mathrm{Tr}\!\left((Q + K^\top R K)\Sigma\right),\\
        s.t.\quad & \Sigma = I + (A + B K) \Sigma (A + B K)^\top
    \end{aligned}
\end{equation}
where $\mathcal{K}$ denotes the stabilizing gain set defined as
\[
\mathcal{K} :=\{K\in\mathbb{R}^{m\times n} :\rho(A+BK)<1\}.
\]
The gradient of $J(K)$ admits a closed-form expression \cite{fazel2018global}:

\begin{align*} \label{eq_lqr_grad_update_single}
\nabla_K J(K)  &= 2\Big[(R + B^\top P B)K + B^\top P A\Big]\Sigma, \\
P &= Q + K^\top R K + (A + B K)^\top P (A + B K), \\
\Sigma &= I + (A + B K) \Sigma (A + B K)^\top.
\end{align*}

This formulation connects LQR to smooth nonconvex optimization and enables the use of gradient-based algorithms. Therefore, the Gradient Descent (GD) update is 
\[K_{i+1}  = K_i - \eta_i \nabla_K J(K_i),\]
where $\eta_i > 0$ is a stepsize. Define the sublevel set 
\[
\mathcal{K}_\gamma :=\{K\in\mathcal{K} :J(K)\leq \gamma\}.
\]
for some finite positive value $\gamma$. It has been proved that policy gradient steps starting from $K_0\in \mathcal{K}_\gamma$ remain inside $\mathcal{K}_\gamma$ and converge to the global optimum of LQR in $\mathcal{K}_\gamma$ at a linear rate \cite{hu2023toward}.

\subsection{Domain Randomization}
In many reinforcement learning and control applications, the system dynamics are uncertain. 
We model this uncertainty by a distribution $\mathcal{D}$ over system matrices $(A,B)$. 
For a controller $K$, the LQR cost corresponding to a realization $(A,B)$ is denoted by $J(K;A,B)$. 
Domain randomization seeks a single controller that minimizes the expected cost over this distribution 
\begin{equation} \label{eq_DRObjective}
J_{\mathrm{DR}}(K) = \mathbb{E}_{(A,B)\sim\mathcal{D}}[J(K;A,B)].
\end{equation}
The goal is therefore
\begin{equation} \label{eq_DRLQR}
K_{\mathrm{DR}} \in \arg\min_{K\in\mathcal{K}_{js}} J_{\mathrm{DR}}(K),
\end{equation}
where $\mathcal{K}_{js}$ denotes the set of jointly stabilizing gains
\[
\mathcal{K}_{js} := \{K\in\mathbb{R}^{m\times n} :
\rho(A+BK)<1,\ \forall (A,B)\in\mathrm{supp}(\mathcal D)\}.
\]

We refer to the problem \eqref{eq_DRLQR} as Domain-Randomized LQR (DR-LQR).
Throughout the paper, we consider the following assumption.
\begin{assumption}[Simultaneous stabilizability]
\label{assump_stabilizable}
The joint stabilizing gain set $\mathcal K_{js}$ is nonempty.
\end{assumption}

%This assumption ensures that the feasible set of \eqref{eq_DRLQR} is nonempty.
This assumption ensures that the feasible set of \eqref{eq_DRLQR} is nonempty. It can, in particular, be satisfied when the heterogeneity among the system dynamics is sufficiently small.

%================================================================
%================================================================
%================================================================

\section{Optimization-Derived Dynamic Programming for LQR} 
\label{Sec_OPPI}
In this section, we connect the dynamic-programming view of LQR with its optimization formulation.

The LQR objective $J(K)$ is gradient dominated over $\mathcal K_\gamma$  \cite{fazel2018global}, so that any stationary point is globally optimal. In particular, the optimal controller satisfies the first-order condition $\nabla_K J(K_\star)=0$ where $K_\star\in \mathcal K_\gamma$.

Now consider a stabilizing iterate $K_i$, and let $P_i$ and $\Sigma_i$ denote the corresponding solutions of the Lyapunov equations. Fixing $P_i$ and $\Sigma_i$, the first-order condition
\[
\nabla_K J(K \mid P_i,\Sigma_i)=0
\]
coincides with the classical policy-improvement equation \eqref{eq_PolicyImprove_single}. Since the resulting subproblem is convex in $K$, the next iterate can be written as
\[
K_{i+1} = \arg\min_K J(K \mid P_i,\Sigma_i).
\]
This yields the following optimization-based formulation, initialized from a stabilizing controller $K_0$:
\begin{align*} 
P_i &= Q + K_i^\top R K_i + (A + B K_i)^\top P_i (A + B K_i), \\
\Sigma_i &= I + (A + B K_i)\Sigma_i(A + B K_i)^\top, \\
K_{i+1} &= \arg\min_K J(K \mid P_i,\Sigma_i). \label{eq_PI_GD}
\end{align*}
That is, classical policy iteration can be interpreted as an alternating optimization scheme over the value and policy variables, closely related to block coordinate descent or alternating minimization methods \cite{bertsekas1999nonlinear,razaviyayn2013unified}.
This perspective will be useful for extending PI to the domain-randomized LQR setting in the next section.

%================================================================
%================================================================
%================================================================

\section{Dynamic Programming for Domain Randomized LQR} \label{sec_DRPI}

We now extend the policy iteration viewpoint developed in the previous section to the domain-randomized setting. 
\subsection{Policy Iteration for DR-LQR}
Domain randomization seeks a single controller that performs well across a distribution of systems $(A,B)\sim\mathcal D$. 
In practice, the distribution is approximated using a finite set of sampled systems. 

\paragraph{Sample-Average Approximation}
Assume that $M$ systems $\{(A_j,B_j)\}_{j=1}^M$ are sampled identically and independently from $\mathcal D$. 
The domain-randomized objective \eqref{eq_DRObjective} can then be approximated by the sample-average objective
\begin{equation}\label{eq_JSA}
    J_{SA}(K)=\frac{1}{M}\sum_{j=1}^M J(K;A_j,B_j).
\end{equation}
For a stabilizing controller $K$, the gradient of $J_{SA}(K)$ is obtained by averaging the gradients of the individual LQR objectives:
\begin{equation}
\begin{aligned}
\nabla_K J_{SA}(K)
  &= \frac{1}{M}\sum_{j=1}^M 
  2\Big[(R + B_j^\top P_j B_j)K + B_j^\top P_j A_j\Big]\Sigma_j, \\
P_j &= Q + K^\top R K + (A_j + B_j K)^\top P_j (A_j + B_j K),\\
\Sigma_j &= I + (A_j + B_j K) \Sigma_j (A_j + B_j K)^\top ,
\end{aligned}
\label{eq_lqr_grad_update}
\end{equation}
for all $j\in[M]$, where $P_j$ and $\Sigma_j$ are the Lyapunov solutions associated with the $j$-th system.

\paragraph{Policy Improvement for DR-LQR}
Following the optimization-based interpretation of policy iteration presented in Section~\ref{Sec_OPPI}, at iteration $i$, given $K_i$ and the matrices $\{P_{j,i},\Sigma_{j,i}\}_{j=1}^M$ obtained from the corresponding Lyapunov equations in \eqref{eq_lqr_grad_update}, we perform the policy improvement step by solving the first-order optimality condition
\[
\nabla_K J_{SA}\!\left(\widehat{K}_{i+1}\mid\{P_{j,i},\Sigma_{j,i}\}_{j=1}^M\right)=0.
\]
Substituting \eqref{eq_lqr_grad_update} yields the matrix equation
\[
\sum_{j=1}^M \Big[(R + B_j^\top P_{j,i} B_j)\widehat K_{i+1} + B_j^\top P_{j,i}A_j\Big]\Sigma_{j,i}=0.
\]
This equation generalizes the classical LQR policy-improvement rule to the multi-system setting. 
While the single-system case yields a closed-form Riccati update, the domain-randomized case couples the controllers across all sampled systems.

\paragraph{Closed-form solution}
The above equation can be written as a Sylvester-type linear matrix equation in $\widehat K_{i+1}$. 
Using vectorization, it admits the closed-form solution
\begin{equation}
\begin{aligned}
\mathrm{vec}\left(\widehat K_{i+1}\right) &=- \left( \sum_{j=1}^M \Sigma_{j,i}\otimes R_{j,i} \right)^{-1}
\mathrm{vec}\!\left( \sum_{j=1}^M \Gamma_{j,i} \right),
\end{aligned}
\label{eq_Generalupdate}
\end{equation}
where $R_{j,i}=R + B_j^\top P_{j,i} B_j$ and $\Gamma_{j,i}=B_j^\top P_{j,i} A_j \Sigma_{j,i}$. Since $R_{j,i}\succ0$ and $\Sigma_{j,i}\succ0$, the existence and uniqueness of the solution are guaranteed.
Equation \eqref{eq_Generalupdate} reveals that the policy update can be interpreted as a weighted aggregation of the individual policy-improvement directions across systems, where the weights depend on the state covariance matrices.

\paragraph{DR-LQR Policy Iteration Algorithm}
Combining the policy evaluation and policy improvement steps yields the policy iteration algorithm for domain-randomized LQR, shown in Algorithm~\ref{Alg_PI_DRLQR}. Starting from a stabilizing controller $K_0$, each iteration first evaluates the value and covariance matrices for all sampled systems. The construction of a jointly stabilizing initial controller is discussed in \cite{AbbasECC2026}. The policy is then updated using \eqref{eq_Generalupdate} together with an appropriate step size. Further details are provided in the next subsection.

\begin{algorithm}[t] \caption{Policy Iteration for DR-LQR}
\label{Alg_PI_DRLQR}
\KwIn{Sample systems $\{(A_j,B_j)\}_{j=1}^M$, stabilizing $K_0$, $Q\succeq 0$, $R\succ 0$, step size $\alpha$}
    \For{ \( i = 0, 1, \dots \) (until convergence)}{
        For all $j\in [M]$ solve:
        \(\:\Sigma_{j,i} = I+(A_j + B_j K_i) \Sigma_{j,i} (A_j + B_j K_i)^\top,\) and
        \(\:P_{j,i} = Q+K_i^\top R K_i+(A_j + B_j K_i)^\top P_{j,i} (A_j + B_j K_i)\)\;
        Compute $\widehat K_{i+1}$ using \eqref{eq_Generalupdate} and set $\Delta_i = \widehat K_{i+1}-K_i$\;
        Update $K_{i+1}=K_i+\alpha\Delta_i$\;
    }
\KwOut{$K_i$}
\end{algorithm}

\subsection{KKT Derivation of the Policy-Iteration}
\label{sec_kkt_dr_lqr}

In this subsection, we show that the policy-improvement step used in our domain-randomized LQR algorithm can be derived from the first-order optimality conditions of a constrained sample-average problem.

\subsubsection{Sample-average DR-LQR formulation}

For a shared linear state-feedback controller $u_t=Kx_t$, the sample-average DR-LQR problem can be written as \cite{AbbasECC2026}
\begin{equation}
\label{eq_app_primal_problem}
\begin{aligned}
\min_{K,\{\Sigma_j\}_{j=1}^M}\quad
& \frac1M\sum_{j=1}^M \Tr\!\big((Q+K^\top R K)\Sigma_j\big)\\
\text{s.t.}\quad
& \Sigma_j = I+(A_j+B_jK)\Sigma_j(A_j+B_jK)^\top,\\&\forall j\in[M].
\end{aligned}
\end{equation}

\subsubsection{Lagrangian}

Introduce symmetric Lagrange multipliers $P_j=P_j^\top$ for the Lyapunov
constraints \eqref{eq_app_primal_problem}. The Lagrangian $\mathcal L(K,\{\Sigma_j, P_j\}_{j=1}^M)$ is
\[
\mathcal L = \frac{1}{M}\sum_{j=1}^M
\Tr\!\Big(\bar{Q}\Sigma_j+P_j\big(I+\bar{A}_j\Sigma_j\bar{A}_j^\top-\Sigma_j\big)\Big),
\]
where $\bar{Q}=Q+K^\top RK$ and $\bar{A}_j=A_j+B_jK$. Using the cyclic property of the trace, the Lagrangian can be written as
\begin{equation}
\label{eq_app_lagrangian_compact}
\mathcal L = \frac1M\sum_{j=1}^M \left[ \Tr(P_j) + \Tr\!\Big(
\bar{Q}+\bar{A}_j^\top P_j\bar{A}_j-P_j
\Big)\Sigma_j
\right].
\end{equation}

\subsubsection{KKT conditions}

We now compute the first-order optimality conditions.

\paragraph{Stationarity with respect to $\Sigma_j$}
Differentiating \eqref{eq_app_lagrangian_compact} with respect to $\Sigma_j$
gives
\[
\nabla_{\Sigma_j}\mathcal L=  Q+K^\top R K+(A_j+B_jK)^\top P_j(A_j+B_jK)-P_j=0.
\]
Therefore, $\forall j\in[M]$,
\begin{equation}
\label{eq_app_p_lyap}
P_j = Q+K^\top R K+(A_j+B_jK)^\top P_j(A_j+B_jK).
\end{equation}
\paragraph{Stationarity with respect to $K$}
\begin{equation}
\label{eq_app_grad_stationarity_raw}
\nabla_K \mathcal L = \frac{2}{M}\sum_{j=1}^M \Big[ RK\Sigma_j+B_j^\top P_j(A_j+B_jK)\Sigma_j \Big] =0.
\end{equation}

\paragraph{Primal feasibility}
The primal constraints are exactly
\begin{equation}
\label{eq_app_primal_feasibility}
\Sigma_j = I+(A_j+B_jK)\Sigma_j(A_j+B_jK)^\top,
\quad \forall j\in[M].
\end{equation}

\subsubsection{From KKT conditions to PI}

At iteration $i$, let $K_i$ be a stabilizing controller.
For each sample $j$, define $P_{j,i}$ and $\Sigma_{j,i}$ as the unique solutions of \eqref{eq_app_p_lyap} and \eqref{eq_app_primal_feasibility}, respectively.
These matrices correspond to the policy-evaluation step for the current controller $K_i$. If one substitutes $(P_{j,i},\Sigma_{j,i})$ into the stationarity condition \eqref{eq_app_grad_stationarity_raw} and solves for $K$, one obtains $\widehat K_{i+1}$ defined in \eqref{eq_Generalupdate} that is the policy-improvement step used in the algorithm.

\subsubsection{Interpretation as a Newton-like policy update}

The exact KKT system couples $K$, $\{P_j\}$, and $\{\Sigma_j\}$ through the nonlinear Lyapunov equations. Solving it directly would require tackling the full nonconvex sample-average problem. Our policy-iteration method instead proceeds by alternating between:
\begin{enumerate}
    \item \textbf{Policy evaluation:} compute $(P_{j,i},\Sigma_{j,i})$ from
    \eqref{eq_app_p_lyap} and \eqref{eq_app_primal_feasibility} for the current $K_i$;
    \item \textbf{Policy improvement:} compute $\widehat K_{i+1}$ from the
    linearized KKT condition \eqref{eq_app_grad_stationarity_raw}.
\end{enumerate}
Thus, the update $\widehat K_{i+1}$ may be viewed as the controller that satisfies the stationarity condition of the Lagrangian while keeping the evaluation quantities $(P_{j,i},\Sigma_{j,i})$ fixed at the current iterate.

Finally, to preserve joint stability across all sampled systems, we do not replace $K_i$ by $\widehat K_{i+1}$ directly. Instead, we take a damped step
\[
K_{i+1} = K_i+\alpha(\widehat K_{i+1}-K_i), \qquad \alpha\in(0,\alpha_{\max} ],
\]
where $\alpha$ is selected so that every closed-loop matrix $A_j+B_jK_{i+1}$ remains Schur stable and the sample-average
cost decreases sufficiently (see Section \ref{sec_Theory}).

The derivation above shows that the policy-improvement equation used in our algorithm is not ad hoc. It arises directly from the KKT stationarity condition
of the constrained sample-average DR-LQR problem, after freezing the policy-evaluation quantities associated with the current stabilizing controller.

%================================================================
%================================================================
%================================================================
\section{Theoretical Analysis} \label{sec_Theory}

We analyze the proposed policy iteration algorithm along three dimensions:
(i) preservation of closed-loop stability during the iterations,
(ii) descent of the sample-average objective, and
(iii) subsequential convergence.
Under an additional gradient-dominance condition, we further obtain a linear convergence rate.

Throughout the analysis we consider the sample-average objective \eqref{eq_JSA} and the level set
\[
\mathcal K_{js,\gamma} := \{K\in\mathcal K_{js} : J_{SA}(K)\le\gamma\}.
\]
Let $K_i\in\mathcal K_{js,\gamma}$ and the policy-improvement direction be $\Delta_i := \widehat K_{i+1}-K_i$, where $\widehat K_{i+1}$ solves the policy-improvement equation derived in \eqref{eq_Generalupdate}.  
The next iterate is obtained by $K_{i+1}=K_i+\alpha\Delta_i$, where $\alpha>0$ is the step size.

Recent work suggests that under moderate system heterogeneity, the sample-average LQR objective exhibits favorable geometric properties, including smoothness and gradient dominance \cite{fujinami2025pg}. Motivated by this, we impose the following assumptions.

\begin{assumption}[Smoothness]
\label{assump_smooth}
$\mathcal K_{js,\gamma}$ is compact and the sample-average objective $J_{SA}$ is continuously differentiable
and $L$-smooth on $\mathcal K_{js,\gamma}$.
\end{assumption}

\begin{assumption}[Gradient dominance]
\label{assump_GD_main}
There exists $\mu>0$ and a minimizer $K_\star\in\mathcal K_{js,\gamma}$
such that
\[
J_{SA}(K)-J_{SA}(K_\star) \le \frac{1}{2\mu}\|\nabla J_{SA}(K)\|_F^2, \quad \forall K\in\mathcal K_{js,\gamma}.
\]
\end{assumption}

%%%%%%%%%%%%%%%%%%%%%%%%%%%%%%%%%%%%%%%%%%%%%%%%%%%%%%%%%%%%
\subsection{Closed-Loop Stability Preservation}

We begin by establishing that the algorithm operates entirely within the stabilizing region. This property is fundamental; without it, the objective $J_{SA}(K)$ may become unbounded or undefined, and subsequent descent and convergence would no longer hold.

\begin{theorem}[Stability Preservation]
\label{thm_stability_main}
Suppose $K_i$ jointly stabilizes all sampled systems.  
Then there exists an $\bar{\alpha}_i>0$ such that $K_i+\alpha\Delta_i$ remains jointly stabilizing for all $\alpha\in[0,\bar{\alpha}_i]$.  
\end{theorem}

\begin{proof}[\textbf{Proof}]
For each $j\in[M]$, define the closed-loop matrix along the update direction
\[
F_{j,i}(\alpha):=A_j + B_j (K_i + \alpha \Delta_i), \qquad \alpha\in\mathbb{R}.
\]
By assumption, $K_i$ is jointly stabilizing, hence
\[
\rho(F_{j,i}(0)) = \rho(A_j + B_j K_i) < 1, \qquad \forall j\in[M].
\]

We first note that the set of Schur-stable matrices
\[
\mathbb S := \{F\in\mathbb{R}^{n\times n} : \rho(F) < 1\}
\]
is open. This follows since the spectral radius $\rho(F)$ is continuous in $F$, and $\mathbb S = \rho^{-1}((-\infty,1))$, where $(-\infty,1)$ is open. Since $F_{j,i}(\alpha)$ depends continuously on $\alpha$, for each fixed $j$ there exists $\bar\alpha_{j,i} > 0$ such that
\[
\rho(F_{j,i}(\alpha)) < 1, \qquad \forall \alpha \in [0,\bar\alpha_{j,i}].
\]

Finally, since the number of systems $M$ is finite, define
\[
\bar{\alpha}_i := \min_{j\in[M]} \bar\alpha_{j,i} > 0.
\]
Then for all $j\in[M]$ and all $\alpha\in[0,\bar{\alpha}_i]$, we have
$\rho(F_{j,i}(\alpha)) < 1$,
which implies that $K_i + \alpha \Delta_i$ jointly stabilizes all systems.  
\end{proof}

While Theorem~\ref{thm_stability_main} establishes the existence of a stabilizing step size at each iteration $i$, a uniform lower bound is required to prevent the step sizes from vanishing asymptotically. Fortunately, this uniform bound arises naturally from the compactness of our sublevel set.

% \begin{lemma}[Uniform Stability Margin]
% \label{lem_uniform_margin}
% Under Assumptions~\ref{assump_stabilizable} and \ref{assump_smooth}, there exists a uniform lower bound $\bar{\alpha} > 0$ such that $\bar{\alpha}_i \ge \bar{\alpha}$ for all $i \ge 0$.
% \end{lemma}
% \begin{proof}
% By Assumption~\ref{assump_smooth}, the sublevel set $\mathcal{K}_{js,\gamma}$ is compact. For any $K \in \mathcal{K}_{js,\gamma}$, the policy evaluation matrices $P_j(K)$ and $\Sigma_j(K)$ are continuous with respect to $K$. Consequently, the policy improvement direction $\Delta(K) = \widehat{K} - K$ is also a continuous mapping on $\mathcal{K}_{js,\gamma}$. 

% For each $K \in \mathcal{K}_{js,\gamma}$ and $j \in [M]$, define the maximum stabilizing step size:
% $$ \bar{\alpha}_{j}(K) := \sup \{ \alpha \ge 0 \mid \rho(A_j + B_j(K + \alpha \Delta(K))) < 1 \} $$
% Because $K$ strictly stabilizes each system, $\bar{\alpha}_{j}(K) > 0$. The spectral radius $\rho$ is a continuous function of its matrix arguments, and $\Delta(K)$ is bounded on the compact set $\mathcal{K}_{js,\gamma}$. Thus, the map $K \mapsto \bar{\alpha}_{j}(K)$ attains a strictly positive minimum over $\mathcal{K}_{js,\gamma}$. 

% Letting $\bar{\alpha} = \min_{j \in [M]} \min_{K \in \mathcal{K}_{js,\gamma}} \bar{\alpha}_{j}(K) > 0$, we guarantee that $\bar{\alpha}_i \ge \bar{\alpha} > 0$ for all iterations $i$, since the sequence $\{K_i\}$ never leaves $\mathcal{K}_{js,\gamma}$.
% \end{proof}

\begin{lemma}[Uniform Stability Margin]
\label{lem_uniform_margin}
Suppose the iterates $\{K_i\}$ lie in a compact subset
$\mathcal K_{js,\gamma} \subset\mathcal K_{js}$, where
\[
\mathcal K_{js} :=\{K:\rho(A_j+B_jK)<1,\ \forall j\in[M]\},
\]
and the update directions satisfy
$L:=\sup_{i\ge0}\|\Delta_i\|_F<\infty$.
Then there exists $\bar\alpha>0$ such that
$K_i+\alpha\Delta_i\in\mathcal K_{js}$ for every
$i\ge0$ and $\alpha\in[0,\bar\alpha]$.
\end{lemma}

\begin{proof}
Continuity of the spectral radius and finiteness of $M$
imply that $\mathcal K_{js}$ is open. Compactness of
$\mathcal K_{js,\gamma} \subset\mathcal K_{js}$ therefore guarantees
$r>0$ such that $K+E\in\mathcal K_{js}$ whenever
$K\in\mathcal K_{js,\gamma}$ and $\|E\|_F<r$.
Choose $\bar\alpha:=r/(1+L)>0$. For every $i\ge0$
and $\alpha\in[0,\bar\alpha]$,
\[
\|\alpha\Delta_i\|_F
\le \bar\alpha L
=\frac{rL}{1+L}<r.
\]
Hence $K_i+\alpha\Delta_i\in\mathcal K_{js}$.
\end{proof}
%%%%%%%%%%%%%%%%%%%%%%%%%%%%%%%%%%%%%%%%%%%%%%%%%%%%%%%%%%%%
\subsection{Descent Property}
Having guaranteed that the iterates remain stabilizing, we next analyze the optimization aspect of the algorithm. In particular, we show that the policy-improvement step defines a descent direction for the objective, which, together with an appropriate step size, leads to monotonic decrease of the sample-average cost.

\begin{theorem}[Descent Direction]
\label{thm_descent_main}
If $\Delta_i\neq0$, then there exists a constant $c_d>0$ such that
$\langle \nabla J_{SA}(K_i),\Delta_i\rangle \le -c_d\|\Delta_i\|_F^2$ .
\end{theorem}
\begin{proof}[\textbf{Proof}]
From the gradient expression of the sample-average objective,
\[
\nabla J_{SA}(K_i) = \frac{2}{M}\sum_{j=1}^M
\bigl(R_{j,i}K_i+B_j^\top P_{j,i}A_j\bigr)\Sigma_{j,i},
\]
where $R_{j,i}:=R+B_j^\top P_{j,i}B_j$. Moreover, the policy-improvement equation
$\nabla_K J_{SA}(\widehat K_{i+1}\mid\{P_{j,i},\Sigma_{j,i}\}_{j=1}^M)=0$
implies
\[
\sum_{j=1}^M \bigl(R_{j,i}\widehat K_{i+1}+B_j^\top P_{j,i}A_j\bigr)\Sigma_{j,i}=0,
\]
that is,
\[
\sum_{j=1}^M B_j^\top P_{j,i}A_j\Sigma_{j,i} = -\sum_{j=1}^M R_{j,i}\widehat K_{i+1}\Sigma_{j,i}.
\]
Substituting this identity into the gradient expression gives
\begin{equation}
\label{eq_gradRep_app}
\nabla J_{SA}(K_i) = -\frac{2}{M}\sum_{j=1}^M R_{j,i}\Delta_i\Sigma_{j,i},
\end{equation}
where \(\Delta_i:=\widehat K_{i+1}-K_i\).

Taking the Frobenius inner product with \(\Delta_i\), we obtain
\[
\langle \nabla J_{SA}(K_i),\Delta_i\rangle = -\frac{2}{M}\sum_{j=1}^M
\Tr\!\bigl(\Delta_i^\top R_{j,i}\Delta_i\Sigma_{j,i}\bigr).
\]

Since \(R\succ0\) and \(P_{j,i}\succeq0\), we have
\[
R_{j,i}=R+B_j^\top P_{j,i}B_j \succeq R \succeq \underline{r} I,
\qquad \underline{r}:=\lambda_{\min}(R)>0.
\]
Also, since \(\Sigma_{j,i}\succeq I\), it follows that for every \(j\),
\[
\Tr\!\bigl(\Delta_i^\top R_{j,i}\Delta_i\Sigma_{j,i}\bigr)
\ge \underline{r}\,\Tr(\Delta_i^\top\Delta_i) = \underline{r}\,\|\Delta_i\|_F^2.
\]
Therefore,
\[
\langle \nabla J_{SA}(K_i),\Delta_i\rangle \le -\frac{2}{M}\sum_{j=1}^M \underline{r}\,\|\Delta_i\|_F^2 = -2\underline{r}\,\|\Delta_i\|_F^2.
\]
Thus the claim holds with $c_d:=2\underline{r} =2\lambda_{\min}(R)$.
\end{proof}
Theorem~\ref{thm_descent_main} shows that the update direction $\Delta_i$ yields a uniform descent direction of the sample-average objective at every iteration. In particular, the bound
\[
\langle \nabla J_{SA}(K_i),\Delta_i\rangle \le -c_d \|\Delta_i\|_F^2
\]
implies that $\Delta_i$ is a strict descent direction whenever $\Delta_i\neq 0$. This property is fundamental for the convergence analysis, as it ensures that the policy-improvement step consistently reduces the objective. The next theorem shows that this descent direction translates into an actual decrease of the objective under a suitable step size.

\begin{theorem}[Monotone Decrease]
\label{thm_Decrease_main} Under Assumptions~\ref{assump_stabilizable}-\ref{assump_smooth} and Lemma~\ref{lem_uniform_margin}, if $\Delta_i\neq0$, and $0<\alpha\leq\min\{\bar{\alpha}, c_d/L\}$, then $J_{SA}(K_i+\alpha \Delta_i) < J_{SA}(K_i)$.
\end{theorem}

\begin{proof}[\textbf{Proof}]
Since $J_{SA}$ is $L$-smooth on $\mathcal K_{js,\gamma}$, for any admissible stable step size $\alpha\in(0,\bar{\alpha}]$ we have
\[
J_{SA}(K_i+\alpha\Delta_i) \le J_{SA}(K_i) + \alpha\langle \nabla J_{SA}(K_i),\Delta_i\rangle + \frac{L}{2}\alpha^2\|\Delta_i\|_F^2.
\]
By Theorem~\ref{thm_descent_main},
$\langle \nabla J_{SA}(K_i),\Delta_i\rangle
\le
- c_d \|\Delta_i\|_F^2$.
Substituting gives
\begin{equation} \label{eq_L_smooth}
    J_{SA}(K_i+\alpha\Delta_i) \le J_{SA}(K_i) + \|\Delta_i\|_F^2 \left( -\alpha c_d+\frac{L}{2}\alpha^2 \right).
\end{equation}
The term in parentheses is negative for all $0<\alpha<2c_d/L$, hence descent holds. 
Moreover, this quadratic is minimized at $\alpha=c_d/L$, which gives the maximum decrease (provided $\alpha\le\bar{\alpha}$ to preserve stability). 
Restricting $\alpha$ to the interval
$0<\alpha\leq\min\{\bar{\alpha}, c_d/L\}$
ensures both stability and strict decrease, completing the proof.
\end{proof}

Together with Theorem~\ref{thm_stability_main}, this result implies that the algorithm generates a sequence of stabilizing controllers with strictly decreasing objective values. %In practice, an appropriate step-size can be achieved using backtracking line search.

%%%%%%%%%%%%%%%%%%%%%%%%%%%%%%%%%%%%%%%%%%%%%%%%%%%%%%%%%%%%
\subsection{Subsequential Convergence}
We now combine the stability preservation and descent properties to establish convergence of the iterates. Intuitively, stability ensures that the iterates remain in a well-defined region, while monotonic decrease prevents oscillations and forces the sequence toward stationary points.

\begin{lemma}[Sufficient decrease]
\label{lem:sufficient_decrease_app}
Let $\alpha=\min\{\bar{\alpha}, c_d/L\}$, and $K_{i+1}=K_i+\alpha\Delta_i$.
Under the assumptions of Theorem~\ref{thm_Decrease_main}, there exists a constant $c_1>0$ such that
\[ 
J_{SA}(K_i)-J_{SA}(K_{i+1}) \ge c_1\|\Delta_i\|_F^2.
\]
\end{lemma}

\begin{proof}[\textbf{Proof}]
From \eqref{eq_L_smooth} we have
\[
J_{SA}(K_{i+1})-J_{SA}(K_i)
\le
-\alpha\left(c_d-\frac{L}{2}\alpha\right)\|\Delta_i\|_F^2.
\]
Thus the claim holds with $c_1:=\alpha\left(c_d-\frac{L}{2}\alpha\right)>0$.
\end{proof}

\begin{theorem}[Subsequential Convergence]
\label{thm_subseq_main} Let the Assumptions~\ref{assump_stabilizable}-\ref{assump_smooth} and Lemma~\ref{lem_uniform_margin} hold, $\alpha=\min\{\bar{\alpha}, c_d/L\}$, $K_0\in\mathcal K_{js,\gamma}$, and $\{K_i\}$ be generated by the proposed algorithm. Then
\begin{enumerate}
\item The sequence $\{K_i\}$ remains in $\mathcal K_{js,\gamma}$ and therefore has accumulation points.

\item The sequence $\{J_{SA}(K_i)\}$ is monotone nonincreasing and convergent.

\item $\|\Delta_i\|_F\to0$ as $i\to\infty$.

\item $\|\nabla J_{SA}(K_i)\|_F\to0$ as $i\to\infty$.

\item Every accumulation point of $\{K_i\}$ is a stationary point of $J_{SA}$.
\end{enumerate}
\end{theorem}

\begin{proof}[\textbf{Proof}]
By Lemma~\ref{lem:sufficient_decrease_app}, there exists $c_1>0$ such that
\[
J_{SA}(K_i)-J_{SA}(K_{i+1}) \ge c_1\|\Delta_i\|_F^2, \qquad \forall i\ge0.
\]
In particular, $J_{SA}(K_{i+1})\le J_{SA}(K_i), \quad \forall i\ge0$,
so the sequence $\{J_{SA}(K_i)\}$ is monotone nonincreasing.

Since $K_0\in\mathcal K_{js,\gamma}$, we have $J_{SA}(K_0)\le\gamma$. By monotonicity,
\[
J_{SA}(K_i)\le J_{SA}(K_0)\le\gamma, \qquad \forall i\ge0.
\]
Hence $K_i\in\mathcal K_{js,\gamma}$, $\forall i\ge0$. Therefore all iterates remain in the compact set $\mathcal K_{js,\gamma}$, and thus the sequence $\{K_i\}$ is bounded and admits accumulation points. This proves the first claim.

Because $J_{SA}$ is continuous on the compact set $\mathcal K_{js,\gamma}$, it is bounded below there. Since $\{J_{SA}(K_i)\}$ is monotone nonincreasing and bounded below, it converges to a finite limit, say
\[
\bar J := \lim_{i\to\infty} J_{SA}(K_i).
\]
This proves the second claim.

Next, summing the sufficient decrease inequality from $i=0$ to $T$ yields
\begin{align*}
c_1\sum_{i=0}^{T}\|\Delta_i\|_F^2
& \le
\sum_{i=0}^{T}\big(J_{SA}(K_i)-J_{SA}(K_{i+1})\big)
\\ &=
J_{SA}(K_0)-J_{SA}(K_{T+1}).
\end{align*}
Since $J_{SA}$ is bounded below on $\mathcal K_{js,\gamma}$, the right-hand side is
uniformly bounded in $T$. Hence
\[
\sum_{i=0}^{\infty}\|\Delta_i\|_F^2<\infty,
\]
which implies $\|\Delta_i\|_F\to0$ as $i\to\infty$. This proves the third claim.

We next show that the gradients vanish.  From \eqref{eq_gradRep_app}, the triangle inequality, and submultiplicativity,
\[
\|\nabla J_{SA}(K_i)\|_F \le
\frac{2}{M}\sum_{j=1}^M \|R_{j,i}\|\,\|\Delta_i\|_F\,\|\Sigma_{j,i}\|,
\]
where \(R_{j,i}:=R+B_j^\top P_{j,i}B_j\). For each fixed \(j\), the matrices
\(P_{j,i}\) and \(\Sigma_{j,i}\) are the unique solutions of the Lyapunov equations
\begin{align*}
    P_{j,i}&=Q+K_i^\top RK_i+(A_j+B_jK_i)^\top P_{j,i}(A_j+B_jK_i),\\
    \Sigma_{j,i}&=I+(A_j+B_jK_i)\Sigma_{j,i}(A_j+B_jK_i)^\top,
\end{align*}
and hence depend continuously on \(K_i\) on the jointly stabilizing set. Since
\(\mathcal K_{js,\gamma}\) is compact, the continuous maps \(K\mapsto \|P_j(K)\|\) and \(K\mapsto \|\Sigma_j(K)\|\) attain their maxima on \(\mathcal K_{js,\gamma}\). Because \(j\in[M]\) is finite, there exist constants \(\bar p,\bar \sigma>0\) such that
\[
\|P_{j,i}\|\le \bar p,\qquad \|\Sigma_{j,i}\|\le \bar \sigma,
\quad
\forall K_i\in\mathcal K_{js,\gamma},\ \forall j\in[M].
\]
Therefore,
\[
\|R_{j,i}\|
=
\|R+B_j^\top P_{j,i}B_j\|
\le
\|R\|+\|B_j\|^2\|P_{j,i}\|
\le
\bar r,
\]
for some constant \(\bar r>0\) independent of \(i\) and \(j\). Consequently,
\begin{equation}\label{eq_gradupperbound}
    \|\nabla J_{SA}(K_i)\|_F
    \le
    2\bar r\,\bar{\sigma}\,\|\Delta_i\|_F.
\end{equation}
Since \(\|\Delta_i\|_F\to0\) as \(i\to\infty\), it follows that \(\|\nabla J_{SA}(K_i)\|_F\to0\). This proves the fourth claim.

Finally, let $\bar K$ be any accumulation point of $\{K_i\}$.
Then there exists a subsequence $\{K_{i_\ell}\}$ such that $K_{i_\ell}\to\bar K$. From the fourth claim, $\|\nabla J_{SA}(K_{i_\ell})\|_F\to0$. Because $J_{SA}$ is continuously differentiable on $\mathcal K_{js,\gamma}$,
its gradient is continuous. Therefore, \[\nabla J_{SA}(\bar K)=\lim_{\ell\to\infty}\nabla J_{SA}(K_{i_\ell})=0.\]
Thus every accumulation point of $\{K_i\}$ is a stationary point of $J_{SA}$. This proves the fifth claim.
\end{proof}

%%%%%%%%%%%%%%%%%%%%%%%%%%%%%%%%%%%%%%%%%%%%%%%%%%%%%%%%%%%%
\subsection{Linear Convergence}

The previous result ensures convergence to stationary points. We now strengthen this result by showing that, under an additional gradient-dominance condition, the convergence is in fact global and occurs at a linear rate.

\begin{theorem}[Linear Convergence]
\label{thm_linear_main}
Under the assumptions of Theorem~\ref{thm_subseq_main} and Assumption~\ref{assump_GD_main},  
there exists $\rho\in(0,1)$ such that for any $K_0\in\mathcal K_{js,\gamma}$ and the minimizer $K_\star\in\mathcal K_{js,\gamma}$, 
\[
J_{SA}(K_i)-J_{SA}(K_\star)
\le
\rho^i\big(J_{SA}(K_0)-J_{SA}(K_\star)\big).
\]
\end{theorem}
\begin{proof}[\textbf{Proof}]
From Lemma~\ref{lem:sufficient_decrease_app}, there exists $c_1>0$ such that
\begin{equation}
\label{eq_suff_dec_again_app}
J_{SA}(K_i)-J_{SA}(K_{i+1})
\ge
c_1\|\Delta_i\|_F^2.
\end{equation}
From \eqref{eq_gradupperbound}, we have the lower-bound of $\|\Delta_i\|_F$ as
\begin{equation}
\label{eq_Delta_lower_by_grad_app}
\|\Delta_i\|_F^2
\ge
\frac{1}{4\bar{r}^2\bar{\sigma}^2}
\|\nabla J_{SA}(K_i)\|_F^2.
\end{equation}
Substituting \eqref{eq_Delta_lower_by_grad_app} into
\eqref{eq_suff_dec_again_app} yields
\[
J_{SA}(K_i)-J_{SA}(K_{i+1})
\ge
\frac{c_1}{4\bar{r}^2\bar{\sigma}^2}
\|\nabla J_{SA}(K_i)\|_F^2.
\]
From Assumption \ref{assump_GD_main},
\[
\|\nabla J_{SA}(K_i)\|_F^2
\ge
2\mu\big(J_{SA}(K_i)-J_{SA}(K_\star)\big),
\]
which results in  
\[
J_{SA}(K_i)-J_{SA}(K_{i+1})
\ge
\frac{c_1\mu}{2\bar{r}^2\bar{\sigma}^2}
\big(J_{SA}(K_i)-J_{SA}(K_\star)\big).
\]
% Define
% \[
% \theta:=
% \frac{c_1\mu}{2\bar{r}^2\bar{\sigma}^2}>0.
% \]
% Then
% \[
% J_{SA}(K_{i+1})-J_{SA}(K_\star)
% \le
% (1-\theta)\big(J_{SA}(K_i)-J_{SA}(K_\star)\big).
% \]
% Since the left-hand side is nonnegative, necessarily $0<\theta\le1$. Set $\rho:=1-\theta\in[0,1)$. Iterating the contraction inequality yields
% \[
% J_{SA}(K_i)-J_{SA}(K_\star)
% \le
% \rho^i\big(J_{SA}(K_0)-J_{SA}(K_\star)\big),
% \qquad
% \forall i\ge0.
% \]
% This completes the proof.

Define
\[
\theta:=
\frac{c_1\mu}{2\bar r^2\bar\sigma^2}>0
\]
Set
\[
\rho:=\max\left\{1-\theta,\frac12\right\}\in(0,1).
\]
Since $J_{SA}(K_i)-J_{SA}(K_\star)\ge0$ and
$1-\theta\le\rho$, the preceding inequality implies
\[
J_{SA}(K_{i+1})-J_{SA}(K_\star)
\le
\rho\bigl(J_{SA}(K_i)-J_{SA}(K_\star)\bigr).
\]
Iterating yields
\[
J_{SA}(K_i)-J_{SA}(K_\star)
\le
\rho^i\bigl(J_{SA}(K_0)-J_{SA}(K_\star)\bigr),
\qquad \forall i\ge0.
\]
This completes the proof.
\end{proof}

This result shows that, beyond mere stationarity, the algorithm achieves a global geometric rate of convergence toward the optimal controller under gradient dominance.

%================================================================
%================================================================
%================================================================

\section{Numerical Analysis} \label{sec_experiment}

In this section, we numerically study the convergence behavior of the proposed policy iteration method for domain-randomized LQR and compare it with Policy Gradient descent \cite{fujinami2025pg} and a semidefinite programming (SDP)-based approach (SDPD) \cite{AbbasECC2026}.

To illustrate the behavior of PI, we consider a two-dimensional inverted pendulum linearized and discretized around the upright equilibrium. The system dynamics are given by
\begin{equation}
    A=\begin{bmatrix}
        1 & \Delta t \\
        \frac{g}{l}\Delta t & 1
    \end{bmatrix}, \qquad
    B=\begin{bmatrix}
        0 \\
        \frac{1}{ml^2}\Delta t
    \end{bmatrix},
\end{equation}
where $\Delta t = 0.01$ and $g = 10$. The mass $m$ and pole length $l$ are uncertain parameters with nominal value $1$, each subject to independent uniform perturbations of $\pm 25\%$.

We approximate the sample-average objective $J_{SA}$ using $M=50$ sampled systems, and the step size are $\alpha=0.1$ and $0.001$ for PI and PG, respectively. The optimal controller $K_\star$ is computed via a dense grid search over the two-dimensional gain space.

Fig.~\ref{fig_trajectory} shows the evolution of the feedback gain starting from $K_0=[-120,\,-10]$. The contour lines correspond to level sets of $J_{SA}(K)$. PG follows directions orthogonal to the level sets, as expected. The SDP-based method exhibits similar behavior when using small perturbation sizes.

In contrast, the PI updates move toward the minimizer of a local quadratic approximation of $J_{SA}$, which can result in more direct progress toward $K_\star$. In this example, the update direction is well aligned with the optimal solution, leading to a nearly straight trajectory toward $K_\star$.
\begin{figure}[t]
    \centering
    \includegraphics[width=0.48\textwidth]{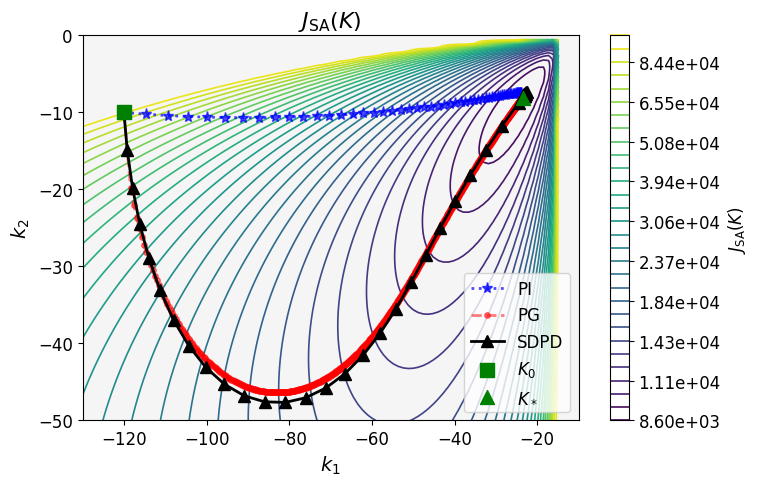}
    \caption{Trajectory of controller updates for PI, PG, and SDPD starting from $K_0=[-120,\,-10]$.}
    \label{fig_trajectory}
\end{figure}
This behavior is not universal and depends on the initial controller. Fig.~\ref{fig_trajectory2} shows the trajectories for initialization $K_0=[-16,\,-40]$. In early iterations, the PI updates follow descent directions that move the controller toward the interior of the stabilizing set rather than directly toward the optimum. As the iterates move away from the boundary, the update direction becomes better aligned with $K_\star$, resulting in accelerated convergence.
\begin{figure}[t]
    \centering
    \includegraphics[width=0.48\textwidth]{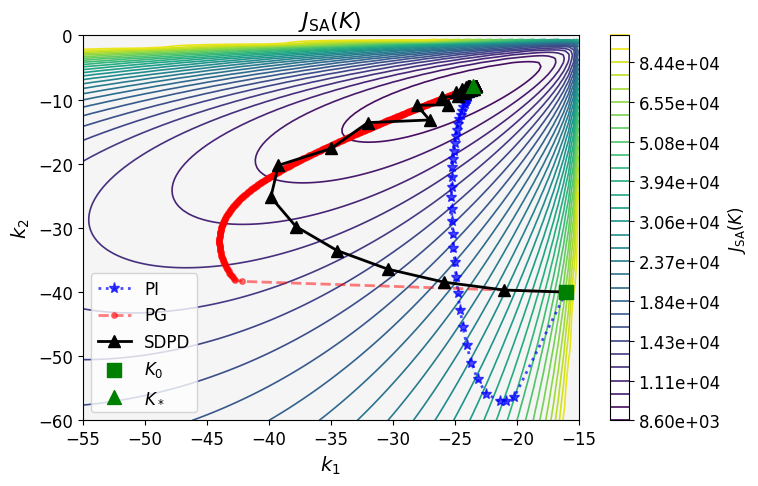}
    \caption{Trajectory of controller updates for PI, PG, and SDPD starting from $K_0=[-16,\,-40]$.}
    \label{fig_trajectory2}
\end{figure}

%================================================================
%================================================================
%================================================================

\section{Conclusion}\label{sec_conclusion}

We studied policy optimization under domain randomization for linear quadratic control to learn a single state-feedback controller that minimizes the average cost across uncertain systems. We demonstrated that our policy iteration algorithm preserves closed-loop stability and monotonically decreases the sample-average objective at each iteration. Furthermore, we established subsequential convergence to stationary points and, under a gradient-dominance condition, proved a global linear convergence rate.

This work provides a theoretical foundation for domain-randomized control, bridging classical LQR and learning-based methods. Future directions include extensions to stochastic or time-varying domains, data-driven settings with finite-sample guarantees, and nonlinear systems.

% Bibliography
\bibliographystyle{IEEEtran}
\bibliography{ref_CDC.bib}

\end{document}